\documentclass[oneside,english]{amsart}
\usepackage[T1]{fontenc}
\usepackage[latin9]{inputenc}
\usepackage[active]{srcltx}
\usepackage{mathtools}
\usepackage{amsbsy}
\usepackage{amstext}
\usepackage{amsthm}
\usepackage{amssymb}
\usepackage{wasysym}

\makeatletter
\numberwithin{equation}{section}
\numberwithin{figure}{section}
\theoremstyle{plain}
\newtheorem{thm}{\protect\theoremname}
\theoremstyle{remark}
\newtheorem*{rem*}{\protect\remarkname}
\theoremstyle{plain}
\newtheorem{cor}[thm]{\protect\corollaryname}
\theoremstyle{plain}
\newtheorem{lem}[thm]{\protect\lemmaname}
\theoremstyle{plain}
\newtheorem*{fact*}{\protect\factname}

\usepackage{babel}

\providecommand{\corollaryname}{Corollary}
\providecommand{\factname}{Fact}
\providecommand{\lemmaname}{Lemma}
\providecommand{\remarkname}{Remark}
\providecommand{\theoremname}{Theorem}

\makeatother

\usepackage{babel}
\providecommand{\corollaryname}{Corollary}
\providecommand{\factname}{Fact}
\providecommand{\lemmaname}{Lemma}
\providecommand{\remarkname}{Remark}
\providecommand{\theoremname}{Theorem}

\begin{document}
\title[The existence of extremal functions for discrete inequalities]{The existence of extremal functions for discrete Sobolev inequalities
on lattice graphs}
\author{Bobo Hua, Ruowei Li}
\address{Bobo Hua: School of Mathematical Sciences, LMNS, Fudan University,
Shanghai 200433, People\textquoteright s Republic of China; Shanghai
Center for Mathematical Sciences, Fudan University, Shanghai 200433,
People\textquoteright s Republic of China}
\email{bobohua@fudan.edu.cn}
\address{Ruowei Li: School of Mathematical Sciences, Fudan University, Shanghai
200433, People\textquoteright s Republic of China; Shanghai Center
for Mathematical Sciences, Fudan University, Shanghai 200433, People\textquoteright s
Republic of China}
\email{rwli19@fudan.edu.cn}
\begin{abstract}
In this paper, we study the existence of extremal functions (pairs)
of the following discrete Sobolev inequality (\ref{eq:sobo1}) and
Hardy-Littlewood-Sobolev inequality (\ref{eq:hls1}) in the lattice
$\mathbb{Z}^{N}$: 
\begin{equation}
{\displaystyle \lVert u\rVert_{\ell^{q}}\leq C_{p,q}\lVert u\rVert_{D^{1,p}},\;\forall u\in D^{1,p}(\mathbb{Z}^{N}),}\label{eq:sobo1}
\end{equation}
where $N\geq3,1\leq p<N,q>p^{\ast}=\dfrac{Np}{N-p},C_{p,q}$ is a
constant depending on $N,p$ and $q$; 
\begin{equation}
\sum_{\substack{i,j\in\mathbb{Z}^{N}\\
i\neq j
}
}\frac{f(i)g(j)}{\mid i-j\mid^{\lambda}}\leq C_{r,s,\lambda}\Vert f\Vert_{\ell^{r}}\Vert g\Vert_{\ell^{s}},\:\forall f\in\ell^{r}(\mathbb{Z}^{N}),g\in\ell^{s}(\mathbb{Z}^{N}),\label{eq:hls1}
\end{equation}
where $r,s>1,$ $0<\lambda<N$, $\frac{1}{r}+\frac{1}{s}+\frac{\lambda}{N}>2$,
$C_{r,s,\lambda}$ is a constant depending on $N,r,s$ and $\lambda$.

We introduce the discrete Concentration-Compactness principle, and
prove the existence of extremal functions (pairs) for the best constants
in the supercritical cases $q>p^{\ast}$ and $\frac{1}{r}+\frac{1}{s}+\frac{\lambda}{N}>2$,
respectively. 
\end{abstract}

\maketitle

\section{Introduction}

For $N\geq3,\,k\geq1,\,p\geq1,\,kp<N,\frac{1}{\,p^{'}}=\frac{1}{p}-\frac{k}{N}$,
we have the classical Sobolev inequality\label{eq:c-Sobo-ineq}, 
\begin{equation}
{\displaystyle \lVert u\rVert_{p^{'}}^{p}\leq C_{p,q}\lVert u\rVert_{D^{k,p}}^{p},\;\forall u\in D^{k,p}\left(\mathbb{R}^{\mathit{N}}\right),}\label{eq:classical}
\end{equation}
where $C_{p,q}$ is a constant depending on $N,p$ and $q$, we omit
the dependence of constants $N$ for convenience, and $D^{k,p}\left(\mathbb{R}^{\mathit{N}}\right)$
denotes the completion of $C_{0}^{\infty}\left(\mathbb{R}^{\mathit{N}}\right)$
in the norm $\lVert u\rVert_{D^{k,p}}^{p}\coloneqq\sum\limits _{|\alpha|=k}\varint\lvert D^{\alpha}u\rvert^{p}\mathbf{\mathrm{dx}}$.

Whether the best constant can be obtained by some $u\in D^{k,p}\left(\mathbb{R}^{\mathit{N}}\right)$,
which is called the extremal function, has been intensively investigated
in the literature. When $k=1,$$\,\,p=1,$ H. Federer, W. Fleming
\cite{FF} and W. Fleming, R. Rishel \cite{FR} proved that the best
constant is the isoperimetric constant and the extremal function is
the characteristic function of a ball. Using Schwarz symmetrization
\cite{PS}, best constants and extremal functions were obtained by
Talenti \cite{T1}, Rodemich \cite{R} and Aubin \cite{A} independently
in the case of $k=1,p>1$. Moreover, Talenti's paper \cite{T2} reveals
the deep relation between isoperimetric inequalites and Sobolev inequalities,
see also Cianchi \cite{C1}. For $k>1,$ P. L. Lions \cite{L1,L2,L3,L4}
established the Concentration-Compactness method, which provided a
new idea for proving the existence of extremal functions. The general
idea is as follows. The best constant in the Sobolev inequality (\ref{eq:classical})
is given by 
\[
S:=\inf_{\substack{u\in D^{k,p}(\mathbb{R}^{N})\\
\lVert u\rVert_{p^{\prime}}=1
}
}\lVert u\rVert_{D^{k,p}}^{p}>0.
\]
Take a minimizing sequence $\{u_{n}\}$ and regard $\left\{ |u_{n}|^{p^{\prime}}\mathrm{dx}\right\} $
as a sequence of probability measures. He proved in \cite{L1,L3}
that there are three cases of the limit of the sequence: compactness,
vanishing and dichotomy. Vanishing and dichotomy are ruled out by
the rescaling trick and subadditivity inequality. Therefore, the extremal
function exists by the compactness. Since the Concentration-Compactness
principle requires weak convergence $u_{n}\xrightharpoonup{}u\;\hphantom{}\text{in}\;D^{k,p}(\mathbb{Z}^{N})$,
this method does not apply the case of $p=1$.

For $r,s>1,$ $0<\lambda<N$, $\frac{1}{r}+\frac{1}{s}+\frac{\lambda}{N}=2$,
we have the classical Hardy-Littlewood-Sobolev (HLS for abbreviation)
inequality in $\mathbb{R}^{\mathit{N}}$ \cite{HL1,HL2,NP}, 
\begin{equation}
\iint\frac{f(x)g(y)}{\mid x-y\mid^{\lambda}}dxdy\leq C_{r,s,\lambda}\Vert f\Vert_{L^{r}}\Vert g\Vert_{L^{s}},\:\forall f\in L^{r}(\mathbb{R}^{\mathit{N}}),g\in L^{s}(\mathbb{R}^{\mathit{N}}).\label{eq:c-hls}
\end{equation}
In \cite{L5}, Lieb proved the existence of the maximizing pair $(f,g)$,
i.e. a pair that gives equality in (\ref{eq:c-hls}). He also gave
the explicit $(f,g)$ and best constant in the case $r=s$. He's method
requires rearrangement inequalities to exclude the vanishing case
and a compactness technique for maximizing sequences \cite[Lemma 2.7]{L5},
which is induced by the Brézis-Lieb lemma \cite{BL}. It also applies
to Sobolev inequalities, doubly weighted HLS inequalities and weighted
Young inequalities. Similar to Sobolev inequalities, Lions \cite{L4}
also proved the existence of extreme functions for HLS inequalities
by the Concentration-Compactness principle.

In recent years, people paid attention to the analysis on graphs.
Since the Sobolev inequalities and HLS inequalities are useful analytical
tools, they have been extended to the discrete setting \cite{CY,HLY,O1}.
For finite graphs, Sobolev inequalities and sharp constants have been
obtained by \cite{NKW,NKYTW,SST,TNK,YKWNT,YWK}. In this article,
we study extremal functions of the discrete Sobolev inequality (\ref{eq:sobo1})
in the lattice $\mathbb{Z}^{N}$. Next, we consider the discrete HLS
inequality (\ref{eq:hls1}). For $N=1$, (\ref{eq:hls1}) is just
the Hardy-Littlewood-P$\acute{\textrm{o}}$lya inequality \cite{HLP}.
In \cite{LV,CL}, the authors considered (\ref{eq:hls1}) in finite
subgraphs of $\mathbb{Z}^{N}$. For the supercritical case $\frac{1}{r}+\frac{1}{s}+\frac{\lambda}{N}>2$,
Huang, Li and Yin \cite{HLY} proved the existence of the maximizing
pair for (\ref{eq:hls1}) on $\mathbb{Z}^{N}$ by analyzing the Euler-Lagrange
equation on bounded subsets with Dirichlet boundary condition and
taking the exhaustion. Inspired by this paper, we also consider (\ref{eq:hls1})
in the supercritical case and get the result by another proof.

A simple and undirected graph $G=(V,E)$ consists of the set of vertices
$V$ and the set of edges $E$. Two vertices $x,y$ are called neighbours,
denoted by $x\thicksim y,$ if there is an edge $e$ connecting $x$
and $y$, i.e. $e=\left\{ x,y\right\} \in E$. In this paper, we mainly
consider integer lattice graphs which serve as the discrete counterparts
of $\mathbb{R}^{N}$. The $N$-dimensional integer lattice graph,
denoted by $\mathbb{Z}^{N}$, is the graph consisting of the set of
vertices $V=\mathbb{Z}^{N}$ and the set of edges 
\[
E=\left\{ \left\{ x,y\right\} :x,y\in\mathbb{Z}^{N},\mathop{\sum\limits _{i=1}^{N}\lvert x_{i}-y_{i}\rvert=1}\right\} .
\]
We denote by $\ell^{p}(\mathbb{Z}^{N})$ the $\ell^{p}$-summable
functions on $\mathbb{Z}^{N}$ and by $D^{1,p}(\mathbb{Z}^{N})$ the
completion of finitely supported functions in the $D^{1,p}$ norm,
see Section 2 for details. The discrete Sobolev inequality (\ref{eq:discrete sobo1})
and HLS inequality (\ref{eq:dhls-1}) in $\mathbb{Z}^{N}$ are well-known,
see \cite[Theorem 3.6]{HM} and \cite{HLY} for proofs: 
\begin{equation}
\lVert u\rVert_{\ell^{p^{\ast}}}\leq C_{p}\lVert u\rVert_{D^{1,p}},\;\forall u\in D^{1,p}(\mathbb{Z}^{N}),\label{eq:discrete sobo1}
\end{equation}
where $N\geq3,1\leq p<N,p^{\ast}=\dfrac{Np}{N-p}$; 
\begin{equation}
\sum_{\substack{i,j\in\mathbb{Z}^{N}\\
i\neq j
}
}\frac{f(i)g(j)}{\mid i-j\mid^{\lambda}}\leq C_{r,s,\lambda}\Vert f\Vert_{\ell^{r}}\Vert g\Vert_{\ell^{s}},\:\forall f\in\ell^{r}(\mathbb{Z}^{N}),g\in\ell^{s}(\mathbb{Z}^{N}),\label{eq:dhls-1}
\end{equation}
where $r,s>1,$ $0<\lambda<N$, $\frac{1}{r}+\frac{1}{s}+\frac{\lambda}{N}=2$.

Since $\ell^{p}(\mathbb{Z}^{N})$ embeds into $\ell^{q}(\mathbb{Z}^{N})$
for any $q>p$, see Lemma \ref{sobo2}, one verifies that the discrete
Sobolev inequality (\ref{eq:discrete sobo1}) and HLS inequality (\ref{eq:dhls-1})
hold when $q\geq p^{\ast}$ and $\frac{1}{r}+\frac{1}{s}+\frac{\lambda}{N}\geq2$
respectively. Recalling the continuous setting, it is called subcritical
for $q<p^{\ast}$ (resp. $\frac{1}{r}+\frac{1}{s}+\frac{\lambda}{N}<2$),
critical for $q=p^{\ast}$ (resp. $\frac{1}{r}+\frac{1}{s}+\frac{\lambda}{N}=2$),
and supercritical for $q>p^{\ast}$ (resp. $\frac{1}{r}+\frac{1}{s}+\frac{\lambda}{N}>2$)
for the Sobolev inequality (resp. HLS inequality). Therefore, (\ref{eq:discrete sobo1})
and (\ref{eq:dhls-1}) hold in both critical and supercritical cases
on $\mathbb{Z}^{N}$.

The optimal constant in the Sobolev inequality (\ref{eq:sobo1}) is
given by 
\begin{equation}
S:=\inf_{\substack{u\in D^{1,p}(\mathbb{Z}^{N})\\
\lVert u\rVert_{q}=1
}
}\lVert u\rVert_{D^{1,p}}^{p}.\label{eq:inf}
\end{equation}
In order to prove that the infimum is achieved, we consider a minimizing
sequence $\{u_{n}\}\subset D^{1,p}(\mathbb{Z}^{N})$ satisfying 
\begin{equation}
\lVert u_{n}\rVert_{q}=1,\lVert u_{n}\rVert_{D^{1,p}}^{p}\longrightarrow S,n\longrightarrow\infty.\label{eq:min seq}
\end{equation}
We want to prove $u_{n}\longrightarrow u$ strongly in $D^{1,p}(\mathbb{Z}^{N})$,
which will imply $u$ is a minimizer.

For the discrete HLS inequality, we first consider the following equivalent
form of (\ref{eq:hls1}),

\begin{equation}
\lVert\mid i\mid^{-\lambda}\ast f\rVert_{t}\leq C_{r,t}\lVert f\rVert_{r},\label{eq:hls2}
\end{equation}
where $\mid i\mid^{-\lambda}\ast f\coloneqq\underset{j\neq i}{\Sigma}\frac{f(j)}{\mid i-j\mid^{\lambda}}$,
$r,r^{\ast},t>1,0<\lambda<N,\frac{1}{r^{\ast}}+\frac{\lambda}{N}=1+\frac{1}{t},r<r^{\ast}.$

The optimal constant in the inequality (\ref{eq:hls2}) is given by
\[
K\coloneqq\underset{\Vert f\Vert_{r}=1}{\sup}\lVert\mid i\mid^{-\lambda}\ast f\rVert_{t}.
\]
We consider a maximizing sequence $\{f_{n}\}$ satisfying 
\begin{equation}
\lVert f_{n}\rVert_{r}=1,\lVert\mid i\mid^{-\lambda}\ast f_{n}\rVert_{t}\longrightarrow K,n\longrightarrow\infty.\label{eq:max squ}
\end{equation}
We want to prove $f_{n}\longrightarrow f$ strongly in $\ell^{r}(\mathbb{Z}^{N})$,
and hence $f$ is a maximizer. Then by Lemma \ref{lem:exist g}, for
$\frac{1}{t}+\frac{1}{s}=1$ there exists $g\in\ell^{s}(\mathbb{Z}^{N})$
with $\lVert g\rVert_{s}=1$ such that $(f,g)$ is a maximizing pair
for (\ref{eq:hls1}).

We prove the following main results. 
\begin{thm}
For $N\geq3,1\leq p<N$, $q>p^{\ast}=\frac{Np}{N-p}$, \label{thm:main1}let
$\left\{ u_{n}\right\} \subset D^{1,p}(\mathbb{Z}^{N})$ be a minimizing
sequence satisfying (\ref{eq:min seq}). Then there exists a sequence
$\{i_{n}\}\subset\mathbb{Z}^{N}$ such that the sequence after translation
$\left\{ v_{n}(i):=u_{n}(i+i_{n})\right\} $ contains a convergent
subsequence that converges to $v$ in $D^{1,p}(\mathbb{Z}^{N})$.
And $v$ is a minimizer for $S$. 
\end{thm}

\begin{rem*}
(1) We prove the case $p=1$ on $\mathbb{Z}^{N}$, while it is not
true in the continuous case. In our case, $\left\{ \lvert\nabla u_{n}\rvert_{1}\right\} $
contains a $w^{\ast}$-convergent subsequence in $\ell^{1}(\mathbb{Z}^{N})$
by the discrete nature. This fails for the continuous case since $L^{1}(\mathbb{R}^{N})$
is not a dual space of any normed linear space.

(2) The best constant can be obtained in the supercritical case. 
\end{rem*}
Let $(\Gamma,S)$ be a Cayley graph of a discrete group $G$ with
a finite generating set $S.$ In particular, $\mathbb{Z}^{N}$ is
a Cayley graph of a free abliean group. For any $r\in\mathbb{N},$
we denote by $V(r)$ the number of group elements with word length
at most $r.$ For a Cayley graph it is well known that if $V(r)\geq Cr^{D},\ \forall r\geq1$
for $D\geq3,$ then the Sobolev inequality holds, 
\begin{equation}
{\displaystyle \lVert u\rVert_{\ell^{q}}\leq C_{p,q}\lVert u\rVert_{D^{1,p}},}\label{eq:sobo10}
\end{equation}
where $1\leq p<D,q\geq\dfrac{Dp}{D-p}.$ In fact, this follows from
a standard trick and the isoperimetric estimate \cite[Theorem 4.18]{W3}.
By the same argument, we can prove the following result. 
\begin{thm}
\label{thm:Cayley}Let $(\Gamma,S)$ be a Cayley graph satisfying
$V(r)\geq Cr^{D},\ \forall r\geq1$ for $D\geq3.$ For $1\leq p<D$,
$q>\frac{Dp}{D-p}.$ Let $\left\{ u_{n}\right\} \subset D^{1,p}$
be a minimizing sequence in \eqref{eq:sobo10} with $\|u_{n}\|_{q}=1.$
Then there exists a sequence $\{g_{n}\}\subset\Gamma$ such that the
sequence after translation $\{v_{n}\}$ with $v_{n}(g):=u_{n}(g_{n}g),g\in\Gamma,$
contains a convergent subsequence that converges to $v$ in $D^{1,p}$
and $v$ is a minimizer for $S$. 
\end{thm}

Similarly, we have the following theorem for the discrete HLS inequality. 
\begin{thm}
\label{thm:main2}For $r,s,t>1,0<\lambda<N,\frac{1}{t}+\frac{1}{s}=1,\frac{1}{r}+\frac{\lambda}{N}>1+\frac{1}{t}.$
Let $\left\{ f_{n}\right\} \subset\ell^{r}(\mathbb{Z}^{N})$ be a
maximizing sequence satisfying (\ref{eq:max squ}). Then there exists
a sequence $\{i_{n}\}\subset\mathbb{Z}^{N}$ such that the sequence
after translation $\left\{ v_{n}(i):=f_{n}(i+i_{n})\right\} $ contains
a convergent subsequence that converges to $f$ in $\ell^{r}(\mathbb{Z}^{N})$.
And $f$ is a maximizer for $K$. Moreover, there exists $g\in\ell^{s}(\mathbb{Z}^{N})$
with $\lVert g\rVert_{s}=1$ such that $(f,g)$ is a maximizing pair
for (\ref{eq:hls1}) in the supercritical case $\frac{1}{r}+\frac{1}{s}+\frac{\lambda}{N}>2$. 
\end{thm}

\begin{rem*}
(1) Unlike Lieb's proof in \cite{L5}, our proof does not depend on
the rearrangement trick and the special properties of $I(i)=\mid i\mid^{-\lambda}$,
namely $I(i)$ is spherically symmetric and decreasing. This enables
us to treat general classes of potentials $I(i)$.

(2) The best constant can be obtained in the supercritical case. This
has been proved by \cite{HLY}. Here we give an alternative proof. 
\end{rem*}
We will provide two proofs for the main results. In the continuous
setting, Lions proved the existence of extremal functions by Concentration-Compactness
principle \cite[Lemma I.1.]{L3} and a rescaling trick \cite[Theorem I.1, (17)]{L3}.
And Lieb in \cite{L5} used a compactness technique and the rearrangement
inequalities. Following Lions, the main idea of proof I is to prove
a discrete analog of Concentration-Compactness principle, see Lemma
\ref{lem:Concentration-Compact}. However, we don't know proper notion
of rescaling and rearrangement tricks on $\mathbb{Z}^{N}$ to exclude
the vanishing case of the limit function. Inspired by \cite{HLY},
for the supercritical case, we prove that the translation sequence
has a uniform positive lower bound at the origin, see Lemma \ref{lem:lower bound},
which excludes the vanishing case. The idea of proof II is based on
a compactness technique by Lieb \cite[Lemma 2.7]{L5} and the nontrivial
nonvanishing of the limit of translation sequence.

According to \cite{HW,C3}, we define the $p$-Laplaican of $u$ for
$p>1$, 
\[
\Delta_{p}u(x)\coloneqq\underset{y\sim x}{\sum}\lvert u(y)-u(x)\rvert^{p-2}\left(u(y)-u(x)\right),
\]
and for $p=1$,
\[
\Delta_{1}u(x)\coloneqq\left\{ \sum\limits _{y\sim x}f_{xy}:f_{xy}=-f_{yx},f_{xy}\in\textrm{sign}(u(y)-u(x))\right\} \text{, sign\ensuremath{\left(x\right)}=\ensuremath{\begin{cases}
\begin{array}{c}
1,\\{}
[-1,1]\\
-1,
\end{array}, & \begin{array}{c}
x>0,\\
x=0,\\
x<0.
\end{array}\end{cases}}}
\]
Similar to the continuous setting \cite{A,T1}, we have the following
corollary. 
\begin{cor}
\label{cor:p-Laplaican}For $N\geq3,1\leq p<N$, $q>p^{\ast}$, there
is a positive solution of the equation 
\begin{equation}
\Delta_{p}u+u^{q-1}=0,\;x\in\mathbb{Z}^{N}.\label{eq:p-Laplaican}
\end{equation}
\end{cor}

By observing that any non-negative solution to (\ref{eq:p-Laplaican})
is bounded, see Lemma \ref{lem:coro1}, we prove the following theorem
using the results in Lin and Wu's papers \cite{LW,W2}. 
\begin{thm}
\label{thm:nonexist sol}For $N\geq3,$ $p=2$, $2<q\leq\frac{2+2N}{N}$,
there does not exist a non-trivial non-negative solution for (\ref{eq:p-Laplaican}). 
\end{thm}

\begin{rem*}
Compared with the continuous case \cite{GS2,GS1,G}, we conjecture
that if $u$ is a non-negative solution of (\ref{eq:p-Laplaican})
in $\mathbb{Z}^{N}$ with $2\leq q<p^{\ast}$, then $u\equiv0$. According
to Lions \cite[Corollary I.1]{L3}, we conjecture that (\ref{eq:p-Laplaican})
has a positive solution when $q=p^{\ast}$. For $p=2$, we don't know
the existence of non-trivial non-negative solutions for (\ref{eq:p-Laplaican})
when $\frac{2+2N}{N}<q\leq2^{\ast}=\frac{2N}{N-2}$. 
\end{rem*}
By Theorem \ref{thm:main2}, we can get the Euler-Lagrange equation
for (\ref{eq:hls1}) as follows, see \cite{CL3,CL4,CL1,CL2,CLO1,CLO2,CLO3}
for continuous setting and \cite{CZ,HLY} for discrete setting. 
\begin{cor}
\label{cor:E-R eq}For $r,s>1$, $0<\lambda<N$, $\frac{1}{r}+\frac{1}{s}+\frac{\lambda}{N}>2,$
there is a pair of positive solution $(f,g)$ of the following Euler-Lagrange
equation for (\ref{eq:hls1}), 
\begin{equation}
\begin{cases}
K\left(f\left(i\right)\right)^{r-1}=\underset{j\neq i}{\sum}\frac{g(j)}{\mid i-j\mid^{\lambda}}\\
K\left(g\left(i\right)\right)^{s-1}=\underset{j\neq i}{\sum}\frac{f(j)}{\mid i-j\mid^{\lambda}}.
\end{cases}\label{eq:E-L eq}
\end{equation}
\end{cor}

The paper is organized as follows. In Section 2, we recall some basic
facts and prove some useful lemmas. In Section 3, we introduce the
Brézis-Lieb lemma and prove the Concentration-Compactness principle
in $\mathbb{Z}^{N}$. In Section 4, we prove a key lemma to exclude
the vanishing case and give the proof %
\mbox{%
I%
} for Theorem \ref{thm:main1}. In Section 5, we give another proof
\mbox{%
II%
} for Theorem \ref{thm:main1} and prove Theorem \ref{thm:main2} in
a similar way.

\section{Preliminary}

Consider integer lattice graph $\mathbb{Z}^{N}$, which is the graph
consisting of the set of vertices $V=\mathbb{Z}^{N}$ and the set
of edges 
\[
E=\left\{ \left\{ x,y\right\} :x,y\in\mathbb{Z}^{N},\mathop{\sum\limits _{i=1}^{N}\lvert x_{i}-y_{i}\rvert=1}\right\} .
\]

We denote the space of functions on $\mathbb{Z}^{N}$ by $C(\mathbb{Z}^{N})$.
For $u\in C(\mathbb{Z}^{N})$, its support set is defined as $\textrm{supp}(u)\coloneqq\{x\in\mathbb{Z}^{N}:u(x)\neq0\}$.
Let $C_{0}(\mathbb{Z}^{N})$ be the set of all functions with finite
support. For any $u\in C(\mathbb{Z}^{N})$, the $\ell^{p}$ norm of
$u$ is defined as 
\[
\lVert u\rVert_{\ell^{p}(\mathbb{Z}^{N})}\coloneqq\begin{cases}
\left(\underset{x\in\mathbb{Z}^{N}}{\sum}\lvert u(x)\rvert^{p}\right)^{1/p} & \text{\ensuremath{0<p<\infty,}}\\
\underset{x\in\mathbb{Z}^{N}}{\sup}\lvert u(x)\rvert & p=\infty.
\end{cases}
\]
The $\ell^{p}(\mathbb{Z}^{N})$ space is defined as 
\[
\ell^{p}(\mathbb{Z}^{N})\coloneqq\left\{ u\in C(\mathbb{Z}^{N}):\lVert u\rVert_{{\ell^{p}(\mathbb{Z}^{N})}}<\infty\right\} .
\]
In this paper, we shall write $\lVert u\rVert_{{\ell^{p}(\mathbb{Z}^{N})}}$
as $\lVert u\rVert_{{p}}$ for convenience, when there is no confusion.

For any $u\in C(\mathbb{Z}^{N})$, we define difference operator for
any $x\sim y$ as

\[
\nabla_{xy}u=u(y)-u(x).
\]
Let 
\[
\lvert\nabla u(x)\rvert_{p}\coloneqq\left(\sum\limits _{y\sim x}\lvert\nabla_{xy}u\rvert^{p}\right)^{1/p}
\]
be the $p$-norm of the gradient of $u$ at $x$.

The $D^{1,p}$ norm of $u$ is given by 
\[
\lVert u\rVert_{D^{1,p}(\mathbb{Z}^{N})}\coloneqq\lVert\lvert\nabla u\rvert_{p}^{p}\rVert_{\ell^{1}(\mathbb{Z}^{N})}^{1/p}=\left(\sum\limits _{x\in\mathbb{Z}^{N}}\sum\limits _{y\sim x}\lvert\nabla_{xy}u\rvert^{p}\right)^{1/p},
\]
and $D^{1,p}(\mathbb{Z}^{N})$ is the completion of $C_{0}\left(\mathbb{Z}^{N}\right)$
in $D^{1,p}$ norm.

The following lemma is well-known, see \cite[Lemma 2.1]{HLY}. 
\begin{lem}
\label{sobo2}Suppose $u\in l^{p}(\mathbb{Z}^{N})$, then $\lVert u\rVert_{q}\leq\lVert u\rVert_{p},\forall q\geq p.$ 
\end{lem}

The combinatorial distance $d$ is defined as $d(x,y)=\text{inf}{\left\{ k:x=x_{0}\sim\cdot\cdot\cdot\sim x_{k}=y\right\} }$,
i.e. the length of the shortest path connecting $x$ and $y$ by assigning
each edge of length one.

Let $\Omega$ be a subset of $\mathbb{Z}^{N}$. We denoted by 
\[
\delta\Omega\coloneqq\left\{ x\in\mathbb{Z}^{N}\backslash\Omega:\exists y\in\Omega,s.t.\;x\sim y\right\} 
\]
the vertex boundary of $\Omega$, possibly an empty set. We set $\bar{\Omega}\coloneqq\Omega\cup\delta\Omega$.

We denoted by $c_{0}(\mathbb{Z}^{N})$ the completion of $C_{0}\left(\mathbb{Z}^{N}\right)$
in $\ell^{\infty}$ norm. Then it is well-known that $\ell^{1}(\mathbb{Z}^{N})=(c_{0}(\mathbb{Z}^{N}))^{\ast}$.
We set 
\[
\lVert\mu\rVert\coloneqq\sup\limits _{u\in c_{0}(\mathbb{Z}^{N}),\lVert u\rVert_{\infty}=1}\langle\mu,u\rangle,\qquad\forall\mu\in\ell^{1}(\mathbb{Z}^{N}).
\]
By definition, 
\[
\mu_{n}\xrightharpoonup{w^{\ast}}\mu\;\text{in}\;\ell^{1}(\mathbb{Z}^{N})\;\text{if and only if}\;\langle\mu_{n},u\rangle\longrightarrow\langle\mu,u\rangle,\forall u\in c_{0}(\mathbb{Z}^{N}).
\]
In the proof, we will use the following facts (see \cite{C2}). 
\begin{fact*}
(a)Every bounded sequence of $\ell^{1}(\mathbb{Z}^{N})$ contains
a $w^{\ast}$-convergent subsequence. \\
 (b)If $\mu_{n}\xrightharpoonup{w^{\ast}}\mu$ in $\ell^{1}(\mathbb{Z}^{N})$,
then ${\mu_{n}}$ is bounded and 
\[
\lVert\mu\rVert\leq\varliminf\limits _{n\to\infty}\lVert\mu_{n}\rVert.
\]
(c)If $\mu\in\ell^{1+}(\mathbb{Z}^{N}):=\left\{ \mu\in\ell^{1}(\mathbb{Z}^{N}):\mu\geq0\right\} $,
then 
\[
\lVert\mu\rVert=\langle\mu,1\rangle.
\]
\end{fact*}
In the functional analysis, the following lemma is well-known, see
\cite{C2} for a proof. 
\begin{lem}
\label{lem:exist g}For $1<p<\infty,$ $\frac{1}{p}+\frac{1}{q}=1$,
let $0\ensuremath{\not\equiv}f\in\ell^{p}(\mathbb{Z}^{N})$, then
there exists unique $g\in\ell^{q}(\mathbb{Z}^{N})$ with $\lVert g\rVert_{q}=1$
satisfying 
\[
\lVert f\rVert_{p}=\underset{i}{\sum}f(i)g(i)=\underset{\lVert h\rVert_{q}\leq1}{\max}\underset{i}{\sum}f(i)h(i).
\]
\end{lem}

$G=(V,E)$ is a locally finite unweighted graph, that is, the degree
$d_{x}\coloneqq\sharp\left\{ y:y\sim x\right\} $ is finite for each
$x\in V$. Then for any $u\in C(V)$, we define the normalized Laplace
as 
\[
\Delta u(x)\coloneqq\sum\limits _{y\sim x}\frac{1}{d_{x}}(u(y)-u(x)).
\]
Then we have the following lemma. 
\begin{lem}
\label{lem:coro1}For a locally finite graph $G=(V,E)$, if $u$ is
a non-negative solution of equation 
\[
-\Delta u=u^{a},\;a>1.
\]
Then $u(x)\leq1,\forall x\in V.$ 
\end{lem}

\begin{proof}
For any $x\in V,$ without loss of generality, we can assume that
$u(x)>0$. Then we get 
\[
u(x)\geq-\sum\limits _{y\sim x}\frac{1}{d_{x}}(u(y)-u(x))=u(x)^{a}.
\]
Hence, $u(x)\leq1.$ 
\end{proof}
By Lemma \ref{lem:coro1}, we can prove Theorem \ref{thm:nonexist sol}
in Section 1. 
\begin{proof}[Proof of Theorem \ref{thm:nonexist sol}]
If $u$ is a non-trivial non-negative solution for (\ref{eq:p-Laplaican}),
we can define $v(t,x)\coloneqq u(x)$, which satisfies the following
heat equation 
\[
\begin{cases}
\begin{array}{c}
v_{t}=\Delta_{2}v+v^{q-1}\\
v(0,x)=u(x)
\end{array} & \begin{array}{c}
\text{in \ensuremath{(0,+\infty)\times\mathbb{Z}^{N},}}\\
\text{in \ensuremath{\mathbb{Z}^{N},}}
\end{array}\end{cases}
\]
where $u(x)$ is bounded by Lemma \ref{lem:coro1}. Using the results
in Lin and Wu's papers \cite{LW,W2}, we know that for $0<N(q-2)<2$,
any non-trivial non-negative solution $v$ is not global, i.e. $v$
blows up in finite time, which yields a contradiction. This proves
the theorem. 
\end{proof}

\section{Concentration-Compactness Principle}

In this section, we prove the discrete Concentration-Compactness principle.
We first introduce a key lemma as follows \cite[Theorem 1]{BL}.

Consider a measure space $(\Omega,\Sigma,\mu)$, which consists of
a set $\Omega$ equipped with a $\sigma$-algebra $\Sigma$ and a
Borel measure $\mu:\Sigma\longrightarrow\left[0,\infty\right]$. 
\begin{lem}
\label{lem:(Discrete-Brezis-Lieb-lemma)}(Brézis-Lieb lemma) Let $(\Omega,\Sigma,\mu)$
be a measure space, $\{u_{n}\}\subset L^{p}(\Omega,\Sigma,\mu)$,
and $0<p<\infty.$ If \\
 (a)$\;\{u_{n}\}$ is uniformly bounded in $L^{p}$, \\
 (b)$\;u_{n}\longrightarrow u,n\longrightarrow\infty$ $\mu$-almost
everywhere in $\Omega$, then 
\begin{equation}
\lim\limits _{n\to\infty}(\lVert u_{n}\rVert_{L^{p}}^{p}-\lVert u_{n}-u\rVert_{L^{p}}^{p})=\lVert u\rVert_{L^{p}}^{p}.\label{eq:weak}
\end{equation}
\end{lem}

\begin{rem*}
(1) The preceding lemma is a refinement of Fatou's Lemma.

(2) Since $\;\{u_{n}\}$ is uniformly bounded in $L^{p}$, passing
to a subsequence if necessary, we have 
\[
\lim\limits _{n\to\infty}\lVert u_{n}\rVert_{p}^{p}=\lim\limits _{n\to\infty}\lVert u_{n}-u\rVert_{p}^{p}+\lVert u\rVert_{p}^{p}.
\]

(3) If $\Omega$ is countable and $\mu$ is a positive measure defined
on $\Omega$, then we get a discrete version of Lemma \ref{lem:(Discrete-Brezis-Lieb-lemma)}. 
\end{rem*}
\begin{cor}
\label{cor:Corollary-1.} Let $\Omega\subset\mathbb{Z}^{N}$, ${\left\{ u_{n}\right\} }\subset D^{1,p}(\mathbb{Z}^{N})$,
and $1\leq p<\infty$. If \\
 (a)$'$ $\left\{ {u_{n}}\right\} $ is uniformly bounded in $D^{1,p}(\mathbb{Z}^{N})$,
\\
 (b)$'$ $u_{n}\longrightarrow u,n\longrightarrow\infty$ pointwise
in $\mathbb{Z}^{N}$, then 
\begin{equation}
\lim\limits _{n\to\infty}\left(\sum\limits _{i\in\Omega}\lvert\nabla u_{n}(i)\rvert_{p}^{p}-\sum\limits _{i\in\Omega}\lvert\nabla(u_{n}-u)(i)\rvert_{p}^{p}\right)=\sum\limits _{i\in\Omega}\lvert\nabla u(i)\rvert_{p}^{p}.\label{eq:gradient}
\end{equation}
\end{cor}

\begin{proof}
We define two directed edge sets as follows 
\[
E_{1}\coloneqq\left\{ e=\left(e_{-},e_{+}\right):e_{\pm}\in\Omega,e_{-}\sim e_{+}\right\} ,
\]
\[
E_{2}\coloneqq\left\{ e=\left(e_{-},e_{+}\right):\ensuremath{(e_{-},e_{+})\in}\Omega\times\ensuremath{\delta\Omega},e_{-}\sim e_{+}\right\} ,
\]
where $E_{1}$ is the set of internal edges of $\bar{\Omega}$, $E_{2}$
is the set of edges that cross the boundary of $\Omega$, and $e_{-}$
and $e_{+}$ are the initial and terminal endpoints of $e$.

Set $\widetilde{E}\coloneqq E_{1}\cup E_{2}$. We define $\phantom{}\overline{u},\mu:\widetilde{E}\longrightarrow\mathbb{R}$,
$\phantom{}\overline{u}(e)=u(e_{+})-u(e_{-})$, $\mu(e)=1$.

Then we get 
\[
\sum\limits _{\Omega}\lvert\nabla u_{n}(i)\rvert_{p}^{p}=\sum\limits _{\,\,E_{1}}\lvert\overline{u}_{n}(e)\rvert^{p}+\sum\limits _{\,\,E_{2}}\lvert\overline{u}_{n}(e)\rvert^{p}=\lVert\overline{u}_{n}\rVert_{\ell^{p}(\widetilde{E},\Sigma,\mu)}^{p}<\infty,
\]
\[
\overline{u}_{n}\longrightarrow\overline{u}\;\textrm{pointwise\;in}\;\widetilde{E}.
\]
For the measure space $(\widetilde{E},\Sigma,\mu)$, by Lemma \ref{lem:(Discrete-Brezis-Lieb-lemma)}
we have 
\[
\lim\limits _{n\to\infty}(\lVert\overline{u}_{n}\rVert_{\ell^{p}(\widetilde{E},\Sigma,\mu)}^{p}-\lVert\overline{u_{n}-u}\rVert_{\ell^{p}(\widetilde{E},\Sigma,\mu)}^{p})=\lVert\overline{u}\rVert_{\ell^{p}(\widetilde{E},\Sigma,\mu)}^{p},
\]
which is equivalent to the equation (\ref{eq:gradient}). 
\end{proof}
In the continuous setting, P. L. Lions \cite{L3}, Bianchi et al.\cite{BCS}
and Ben-Naoum et al.\cite{BTW} proved that the limit of the minimizing
sequence norm can be divided into three parts, i.e. the norm of the
limit, the norm of the limit of the difference between the sequence
and the limit, and the norm of the sequence at infinity. The corresponding
parts still satisfy the Sobolev inequalities, also see \cite[Lemma 1.40]{W}.
Next, we establish the Concentration-Compactness principle in the
lattice $\mathbb{Z}^{N}$. 
\begin{lem}
\label{lem:Concentration-Compact}(Discrete Concentration-Compactness
lemma) For $N\geq3,1\leq p<N,q\geq p^{\ast}$, if $\{u_{n}\}$ is
uniformly bounded in $D^{1,p}(\mathbb{Z}^{N})$. Then passing to a
subsequence if necessary, still denoted as $\{u_{n}\}$, we have 
\begin{equation}
u_{n}\longrightarrow u\quad\textrm{pointwise\;in}\;\mathbb{Z}^{N},\label{eq:pointwise}
\end{equation}

\begin{equation}
\lvert\nabla u_{n}\rvert_{p}^{p}\xrightharpoonup{w^{\ast}}\lvert\nabla u\rvert_{p}^{p}\quad\textrm{in}\;\ell^{1}(\mathbb{Z}^{N}).\label{eq:w-star}
\end{equation}
And the following limits 
\[
\lim\limits _{R\to\infty}\lim\limits _{n\to\infty}\sum\limits _{d(i,0)>R}\lvert\nabla u_{n}\left(i\right)\rvert_{p}^{p}\coloneqq\mu_{\infty},\:\lim\limits _{R\to\infty}\lim\limits _{n\to\infty}\sum\limits _{d(i,0)>R}\lvert u_{n}(i)\rvert^{q}\coloneqq\nu_{\infty},
\]
exist. For the above $\left\{ u_{n}\right\} $, we have 
\begin{equation}
\lvert\nabla(u_{n}-u)\rvert_{p}^{p}\xrightharpoonup{w^{\ast}}0\quad\text{in}\;\ell^{1}(\mathbb{Z}^{N}),\label{eq:nu=00003D00003D00003D00003D00003D00003D0}
\end{equation}
\begin{equation}
\lvert u_{n}-u\rvert^{q}\xrightharpoonup{w^{\ast}}0\quad\text{in}\;\ell^{1}(\mathbb{Z}^{N}),\label{eq:mu=00003D00003D00003D00003D00003D00003D0}
\end{equation}
\begin{equation}
\nu_{\infty}^{p/q}\leq S^{-1}\mu_{\infty},\label{eq:infinite}
\end{equation}
\begin{equation}
\lim\limits _{n\to\infty}\lVert u_{n}\rVert_{D^{1,p}}^{p}=\lVert u\rVert_{D^{1,p}}^{p}+\mu_{\infty},\label{eq:composition1}
\end{equation}
\begin{equation}
\lim\limits _{n\to\infty}\lVert u_{n}\rVert_{q}^{q}=\lVert u\rVert_{q}^{q}+\nu_{\infty}.\label{eq:composition2}
\end{equation}
\end{lem}

\begin{proof}
Since $\{u_{n}\}$ is uniformly bounded in $\ell^{q}(\mathbb{Z}^{N})$,
and hence in $\ell^{\infty}(\mathbb{Z}^{N})$. By diagonal principle,
passing to a subsequence we get (\ref{eq:pointwise}). Since $\left\{ \lvert\nabla u_{n}\rvert_{p}^{p}\right\} $
is uniformly bounded in $\ell^{1}(\mathbb{Z}^{N})$, we get (\ref{eq:w-star})
by the Banach-Alaoglu theorem and (\ref{eq:pointwise}). For every
$R\geq1$, passing to a subsequence if necessary, 
\[
\lim\limits _{n\to\infty}\sum\limits _{d(i,0)>R}\lvert\nabla u_{n}\left(i\right)\rvert_{p}^{p},\qquad\lim\limits _{n\to\infty}\sum\limits _{d(i,0)>R}\lvert u_{n}(i)\rvert^{q}.
\]
exist, where $d$ is the combinatorial distance as defined in Section
2. Then we can define $\mu_{\infty}$, $\nu_{\infty}$ by the monotonicity
in $R$.

Let $v_{n}:=u_{n}-u$, then $v_{n}\longrightarrow0$ pointwise in
$\mathbb{Z}^{N}$ and $\left\{ \lvert\nabla v_{n}\rvert_{p}^{p}\right\} $
is uniformly bounded in $\ell^{1}(\mathbb{Z}^{N})$. Then any subsequence
of $\left\{ \lvert\nabla v_{n}\rvert_{p}^{p}\right\} $ contains a
subsequence (still denoted as $\left\{ \lvert\nabla v_{n}\rvert_{p}^{p}\right\} $)
that $w^{\ast}$-converges to 0 in $\ell^{1}(\mathbb{Z}^{N})$, which
follows from 
\[
\sum h\lvert\nabla v_{n}\rvert_{p}^{p}\longrightarrow0,\text{ }\forall h\in C_{0}(\mathbb{Z}^{N}).
\]
Hence we get (\ref{eq:nu=00003D00003D00003D00003D00003D00003D0}).
Similarly, we get (\ref{eq:mu=00003D00003D00003D00003D00003D00003D0}).

For $R\geq1$, let $\Psi_{R}\in C(\mathbb{Z}^{N})$ such that $\Psi_{R}(i)=1$
for $d(i,0)\geq R+1$, $\Psi_{R}(i)=0$ for $d(i,0)\leq R$. By the
discrete Sobolev inequality (\ref{eq:sobo1}), we have 
\[
(\sum\lvert\Psi_{R}v_{n}\rvert^{q})^{p/q}\leq S^{-1}\sum\lvert\nabla(\Psi_{R}v_{n})\rvert_{p}^{p}=S^{-1}\underset{i}{\sum}\underset{j\sim i}{\sum}\lvert\nabla_{ij}\Psi_{R}v_{n}(j)+\Psi_{R}(i)\nabla_{ij}v_{n})\rvert^{p}.
\]
And for any $\varepsilon>0$, there exists $C_{\varepsilon}>0$ such
that 
\[
\lvert\nabla_{ij}\Psi_{R}v_{n}(j)+\Psi_{R}(i)\nabla_{ij}v_{n})\rvert^{p}\leq C_{\varepsilon}\lvert\nabla_{ij}\Psi_{R}\lvert^{p}\lvert v_{n}(j)\lvert^{p}+(1+\varepsilon)\lvert\nabla_{ij}v_{n}\lvert^{p}\lvert\Psi_{R}(i)\lvert^{p}.
\]
Since $v_{n}\longrightarrow0$ pointwise in $\mathbb{Z}^{N}$, by
$\varepsilon\longrightarrow0^{+}$ we obtain 
\begin{equation}
\varlimsup\limits _{n\to\infty}(\sum\lvert\Psi_{R}v_{n}\rvert^{q})^{p/q}\leq S^{-1}\varlimsup\limits _{n\to\infty}\sum\lvert\nabla v_{n}\rvert_{p}^{p}\Psi_{R}^{p}.\label{eq:c-c1}
\end{equation}
From the definition of $\Psi_{R}$, we have 
\begin{equation}
\lim\limits _{R\to\infty}\varlimsup\limits _{n\to\infty}\sum\limits _{d(i,0)>R}\lvert\nabla v_{n}(i)\rvert_{p}^{p}=\lim\limits _{R\to\infty}\varlimsup\limits _{n\to\infty}\sum\lvert\nabla v_{n}\rvert_{p}^{p}\Psi_{R}^{p},\label{eq:c-c2}
\end{equation}
\begin{equation}
\lim\limits _{R\to\infty}\varlimsup\limits _{n\to\infty}\sum\limits _{d(i,0)>R}\lvert v_{n}(i)\rvert^{q}=\lim\limits _{R\to\infty}\varlimsup\limits _{n\to\infty}\sum\lvert v_{n}\rvert^{q}\Psi_{R}^{q}.\label{eq:c-c3}
\end{equation}
By Lemma \ref{lem:(Discrete-Brezis-Lieb-lemma)} and Corollary \ref{cor:Corollary-1.},
we have 
\[
\lim\limits _{n\to\infty}\left(\sum\limits _{d(i,0)>R}\lvert\nabla u_{n}(i)\rvert_{p}^{p}-\sum\limits _{d(i,0)>R}\lvert\nabla v_{n}(i)\rvert_{p}^{p}\right)=\sum\limits _{d(i,0)>R}\lvert\nabla u(i)\rvert_{p}^{p},
\]
\[
\lim\limits _{n\to\infty}\left(\sum\limits _{d(i,0)>R}\lvert u_{n}(i)\rvert^{q}-\sum\limits _{d(i,0)>R}\lvert v_{n}(i)\rvert^{q}\right)=\sum\limits _{d(i,0)>R}\lvert u(i)\rvert^{q}.
\]
Hence by the above equalities, we get
\begin{equation}
\lim\limits _{R\to\infty}\lim\limits _{n\to\infty}\sum\limits _{d(i,0)>R}\lvert\nabla v_{n}(i)\rvert_{p}^{p}=\mu_{\infty},\label{eq:c-c4}
\end{equation}
\begin{equation}
\lim\limits _{R\to\infty}\lim\limits _{n\to\infty}\sum\limits _{d(i,0)>R}\lvert v_{n}(i)\rvert^{q}=\nu_{\infty}.\label{eq:c-c5}
\end{equation}
Combining the equations (\ref{eq:c-c1}), (\ref{eq:c-c2}), (\ref{eq:c-c3}),
(\ref{eq:c-c4}) and (\ref{eq:c-c5}), we get 
\[
\nu_{\infty}^{p/q}\leq S^{-1}\mu_{\infty}.
\]

Since $u_{n}\longrightarrow u$ pointwise in $\mathbb{Z}^{N}$, then
for every $R\geq1$, we have 
\begin{align*}
\varlimsup\limits _{n\to\infty}\sum\lvert\nabla u_{n}\rvert_{p}^{p} & =\varlimsup\limits _{n\to\infty}\left(\sum\Psi_{R}\lvert\nabla u_{n}\rvert_{p}^{p}+\sum(1-\Psi_{R})\lvert\nabla u_{n}\rvert_{p}^{p}\right)\\
 & =\varlimsup\limits _{n\to\infty}\sum\Psi_{R}\lvert\nabla u_{n}\rvert_{p}^{p}+\sum(1-\Psi_{R})\lvert\nabla u\rvert_{p}^{p}
\end{align*}
and 
\begin{align*}
\varlimsup\limits _{n\to\infty}\sum\lvert u_{n}\rvert^{q} & =\varlimsup\limits _{n\to\infty}\left(\sum\Psi_{R}\lvert u_{n}\rvert^{q}+\sum(1-\Psi_{R})\lvert u_{n}\rvert^{q}\right)\\
 & =\varlimsup\limits _{n\to\infty}\sum\Psi_{R}\lvert u_{n}\rvert^{q}+\sum(1-\Psi_{R})\lvert u\rvert^{q}.
\end{align*}
Letting $R\longrightarrow\infty$, we obtain 
\[
\lim\limits _{n\to\infty}\sum\lvert\nabla u_{n}\rvert_{p}^{p}=\mu_{\infty}+\sum\lvert\nabla u\rvert_{p}^{p}=\mu_{\infty}+\lVert u\rVert_{D^{1,p}}^{p},
\]
\[
\lim\limits _{n\to\infty}\sum\lvert u_{n}\rvert^{q}=\nu_{\infty}+\sum\lvert u\rvert^{q}=\nu_{\infty}+\lVert u\rVert_{q}^{q}.
\]
\end{proof}
\begin{rem*}
(1) We prove the case $p=1$ on $\mathbb{Z}^{N}$, while it is not
true in the continuous case, since $L^{1}(\mathbb{R}^{N})$ is not
a dual space of any normed linear space.

(2) The difference between the sequence and the limit $w^{\ast}$-converges
to $0$ in $\ell^{1}(\mathbb{Z}^{N})$, i.e. (\ref{eq:nu=00003D00003D00003D00003D00003D00003D0})
and (\ref{eq:mu=00003D00003D00003D00003D00003D00003D0}), which is
not true in continuous setting. For example, consider the sequence
of probability measures $\left\{ \delta_{n}\right\} $ in $[0,1]$,
where $\delta_{n}(x)\coloneqq n\chi_{[0,\frac{1}{n}]}\textrm{d}x$,
then $\delta_{n}\longrightarrow0$ almost everywhere in $[0,1]$.
However, $\delta_{n}\xrightharpoonup{w^{\ast}}\delta_{0}$ in $\left(C[0,1]\right)^{\ast}$
and the Dirac measure $\delta_{0}$ is non-zero. 
\end{rem*}

\section{Proof %
\mbox{%
I%
} for Theorem \ref{thm:main1}}

In this section, we will prove the existence of the extremal function
for the discrete Sobolev inequality (\ref{eq:sobo1}). Firstly, we
prove that the minimizing sequence after translation has a uniform
positive lower bound at the origin. This is crucial to rule out the
vanishing case of the limit function. 
\begin{lem}
\label{lem:lower bound}For $N\geq3,1\leq p<N,q>p^{\ast}$, let $\left\{ u_{n}\right\} \subset D^{1,p}(\mathbb{Z}^{N})$
be a minimizing sequence satisfying (\ref{eq:min seq}). Then $\varliminf\limits _{n\to\infty}\lVert u_{n}\rVert_{\ell^{\infty}}>0$. 
\end{lem}

\begin{proof}
Choosing $q'$ such that $p^{\ast}<q'<q<\infty$, by interpolation
inequality we have 
\[
1=\lVert u_{n}\rVert_{q}^{q}\leq\lVert u_{n}\rVert_{q'}^{q'}\lVert u_{n}\rVert_{\infty}^{q-q'}\leq C_{q',p}^{q'}\lVert u_{n}\rVert_{D^{1,p}}^{q'}\lVert u_{n}\rVert_{\infty}^{q-q'},
\]
where $C_{q',p}$ is the constant in the Sobolev inequality (\ref{eq:sobo1}).

By taking the limit, we obtain 
\[
1\leq C_{q',p}^{q'}S^{\frac{q'}{p}}\varliminf\limits _{n\to\infty}\lVert u_{n}\rVert_{\infty}^{q-q'}.
\]
This proves the lemma. 
\end{proof}
\begin{rem*}
The maximum of $\lvert u_{n}\rvert$ is attainable since $\lVert u_{n}\rVert_{q}=1$.
Define $v_{n}\left(i\right):=u_{n}(i+i_{n})$, where $\lvert u_{n}(i_{n})\rvert=\max\limits _{i}\lvert u_{n}(i)\rvert$.
Then the translation sequence $\left\{ v_{n}\right\} $ is uniformly
bounded in $D^{1,p}(\mathbb{Z}^{N})$, $\lVert v_{n}\rVert_{q}=1$
and $\lvert v_{n}(0)\rvert=\lVert u_{n}\rVert_{\infty}$. By Lemma
\ref{lem:lower bound}, passing to a subsequence if necessary, we
have 
\[
v_{n}\longrightarrow v\quad\;\textrm{pointwise\;in}\;\mathbb{Z}^{N\text{}},
\]
\begin{equation}
\lvert v(0)\rvert=\varliminf\limits _{n\to\infty}\lVert u_{n}\rVert_{\ell^{\infty}}>0.\label{eq:v(0)>0}
\end{equation}
\end{rem*}
Similarly, we have the following corollary for the discrete HLS inequality
(\ref{eq:hls2}). 
\begin{cor}
\label{cor:lower bound2}For $r,t>1,0<\lambda<N,\frac{1}{r}+\frac{\lambda}{N}>1+\frac{1}{t}$,
let $\left\{ f_{n}\right\} \subset\ell^{r}(\mathbb{Z}^{N})$ be a
maximizing sequence satisfying (\ref{eq:max squ}). Then $\varliminf\limits _{n\to\infty}\lVert f_{n}\rVert_{\ell^{\infty}}>0$. 
\end{cor}

\begin{proof}
Choosing $r'$ such that $r<r'<r^{\ast}<\infty$, by interpolation
inequality we have 
\[
C_{r',t}^{-r'}\lVert\mid i\mid^{-\lambda}\ast f_{n}\rVert_{t}^{r'}\leq\lVert f_{n}\rVert_{r'}^{r'}\leq\lVert f_{n}\rVert_{r}^{r}\lVert f_{n}\rVert_{\infty}^{r'-r}=\lVert f_{n}\rVert_{\infty}^{r'-r},
\]
where $C_{r',t}$ is the constant in the HLS inequality (\ref{eq:hls2}).

By taking the limit, we obtain 
\[
C_{r',t}^{-r'}K^{r'}\leq\varliminf\limits _{n\to\infty}\lVert f_{n}\rVert_{\ell^{\infty}}^{r'-r}.
\]
This proves the corollary. 
\end{proof}
Next, we give the proof %
\mbox{%
I%
} of Theorem 1. 
\begin{proof}[Proof %
\mbox{%
I%
} of Theorem 1]
Let $\left\{ u_{n}\right\} \subset D^{1,p}(\mathbb{Z}^{N})$ be a
minimizing sequence satisfying (\ref{eq:min seq}). And the translation
sequence $\left\{ v_{n}\right\} $ is defined in the Remark after
Lemma \ref{lem:lower bound}.

By equalities (\ref{eq:composition2}) and (\ref{eq:composition1})
in Lemma \ref{lem:Concentration-Compact}, passing to a subsequence
if necessary, we get 
\[
S=\lim\limits _{n\to\infty}\lVert v_{n}\rVert_{D^{1,p}}^{p}=\lVert v\rVert_{D^{1,p}}^{p}+\mu_{\infty},
\]
\[
1=\lim\limits _{n\to\infty}\lVert v_{n}\rVert_{q}^{q}=\lVert v\rVert_{q}^{q}+\nu_{\infty}.
\]
From the Sobolev inequality, (\ref{eq:infinite}) and the inequality
\begin{equation}
\left(a^{q}+b^{q}\right)^{p/q}\leq a^{p}+b^{p}\text{,}\;\forall a,b\geq0,\label{eq:trivial inequality}
\end{equation}
we get 
\begin{align*}
S & =\lVert v\rVert_{D^{1,p}}^{p}+\mu_{\infty}\\
 & \geq S((\lVert v\rVert_{q}^{q})^{p/q}+\nu_{\infty}^{p/q})\\
 & \geq S(\lVert v\rVert_{q}^{q}+\nu_{\infty})^{p/q}=S.
\end{align*}
Since $\left(a^{q}+b^{q}\right)^{p/q}<a^{p}+b^{p}$ unless $a=0$
or $b=0$, we deduce from (\ref{eq:v(0)>0}) that $\lVert v\rVert_{q}^{q}=1$.
By 
\[
\lVert v\rVert_{D^{1,p}}^{p}\geq S\lVert v\rVert_{q}^{p},
\]
we get 
\[
\lVert v\rVert_{D^{1,p}}^{p}=S=\lim\limits _{n\to\infty}\lVert v_{n}\rVert_{D^{1,p}}^{p}.
\]
That is, $v$ is a minimizer. 
\end{proof}
Now we are ready to prove Theorem \ref{thm:Cayley}. 
\begin{proof}[Proof of Theorem~\ref{thm:Cayley}]
By the same argument as in Lemma~\ref{lem:lower bound}, we can
show that 
\[
\varliminf\limits _{n\to\infty}\lVert u_{n}\rVert_{\ell^{\infty}}>0.
\]
Let $g_{n}\in\Gamma$ such that $|u_{n}(g_{n})|=\underset{g\in\Gamma}{\max}|u_{n}(g)|.$
Then for $v_{n}:=u_{n}(g_{n}g),$ we can prove the result verbatim
as in Theorem~1. 
\end{proof}

\section{Proof %
\mbox{%
II%
} for Theorem \ref{thm:main1} and Theorem \ref{thm:main2}}

In this section, we give another proof for Theorem \ref{thm:main1}
using the discrete Brézis-Lieb lemma. Then we prove Theorem \ref{thm:main2}
in a similar way. 
\begin{proof}[Proof %
\mbox{%
II%
} of Theorem \ref{thm:main1}]
Using Lemma \ref{lem:lower bound}, by the translation and taking
a subsequence if necessary as before, we can get a minimizing sequence
$\left\{ u_{n}\right\} $ satisfying (\ref{eq:min seq}), $u_{n}\longrightarrow u$
pointwise in $\mathbb{Z}^{N}$, and $\lvert u(0)\rvert>0$.

By Lemma \ref{lem:(Discrete-Brezis-Lieb-lemma)}, the inequality (\ref{eq:trivial inequality})
and the Sobolev inequality, then passing to a subsequence if necessary,
we have 
\begin{align}
S=\underset{n\rightarrow\infty}{\lim}\lVert u_{n}\rVert_{D^{1,p}}^{p} & =\underset{n\rightarrow\infty}{\lim}\frac{\lVert u_{n}\rVert_{D^{1,p}}^{p}}{\lVert u_{n}\rVert_{q}^{p}}=\underset{n\rightarrow\infty}{\overline{\lim}}\frac{\lVert u_{n}-u\rVert_{D^{1,p}}^{p}+\lVert u\rVert_{D^{1,p}}^{p}}{\left(\lVert u_{n}-u\rVert_{q}^{q}+\lVert u\rVert_{q}^{q}\right)^{p/q}}\label{eq:p/q-1}\\
 & \geq\underset{n\rightarrow\infty}{\overline{\lim}}\frac{\lVert u_{n}-u\rVert_{D^{1,p}}^{p}+\lVert u\rVert_{D^{1,p}}^{p}}{\lVert u_{n}-u\rVert_{q}^{p}+\lVert u\rVert_{q}^{p}}\nonumber \\
 & \geq\underset{n\rightarrow\infty}{\overline{\lim}}\frac{S\lVert u_{n}-u\rVert_{q}^{p}+\lVert u\rVert_{D^{1,p}}^{p}}{\lVert u_{n}-u\rVert_{q}^{p}+\lVert u\rVert_{q}^{p}}.\nonumber 
\end{align}
Since $u\text{\ensuremath{\not\equiv}}0$, we have that 
\[
\lVert u\rVert_{D^{1,p}}^{p}\leq S\lVert u\rVert_{q}^{p},
\]
which implies 
\[
\lVert u\rVert_{D^{1,p}}^{p}=S\lVert u\rVert_{q}^{p}.
\]
By (\ref{eq:p/q-1}), passing to a subsequence, we get
\[
\underset{n\rightarrow\infty}{\lim}\lVert u_{n}-u\rVert_{D^{1,p}}^{p}=S\underset{n\rightarrow\infty}{\lim}\lVert u_{n}-u\rVert_{q}^{p}.
\]
Since $0<\lVert u\rVert_{q}\leq\underset{n\rightarrow\infty}{\lim}\lVert u_{n}\rVert_{q}=1$,
it suffices to show that $\lVert u\rVert_{q}=1$. Suppose that it
is not true, i.e. $0<\lVert u\rVert_{q}=D<1$, then by Lemma \ref{lem:(Discrete-Brezis-Lieb-lemma)},
\[
\underset{n\rightarrow\infty}{\lim}\lVert u_{n}-u\rVert_{q}^{q}=\underset{n\rightarrow\infty}{\lim}\lVert u_{n}\rVert_{q}^{q}-\lVert u\rVert_{q}^{q}=1-D^{q}>0.
\]
However, $\left(a^{q}+b^{q}\right)^{p/q}<a^{p}+b^{p}$ if $a,$ $b>0$.
This yields a contradiction by (\ref{eq:p/q-1}).

Thus, $\lVert u\rVert_{q}=1$ and $u$ is a minimizer. 
\end{proof}
Taking measure spaces $(M,\varSigma,\mu)$ and $(M^{\prime},\varSigma^{\prime},\mu^{\prime})$
in \cite[Lemma 2.7]{L5} as $\mathbb{Z}^{N}$, we have the following
lemma. 
\begin{lem}
\label{lem:proof2}Let $A$ be a bounded linear operator from $\ell^{p}(\mathbb{Z}^{N})$
to $\ell^{q}(\mathbb{Z}^{N})$ with $1\leq p\leq q<\infty$. For $u\in\ell^{p}(\mathbb{Z}^{N})$,
$u\text{\ensuremath{\not\equiv}}0$, let 
\[
R(u)\coloneqq\frac{\lVert Au\rVert_{q}}{\lVert u\rVert_{p}},\text{ and \ensuremath{N\coloneqq\textrm{sup}\left\{ R(u)\mid u\in\ell^{p}(\mathbb{Z}^{N}),u\ensuremath{\not\equiv}0\right\} .} }
\]
Let $\left\{ u_{n}\right\} $ be a uniformly normed-bounded maximizing
sequence for $N$. Suppose that 
\[
u_{n}\longrightarrow u\ensuremath{\not\equiv}0\text{ pointwise,}
\]
\[
Au_{n}\longrightarrow Au\text{ pointwise .}
\]
Then $u$ is a maximizer, i.e. $R(u)=N$.

Moreover, if $p<q$ and $\underset{n\rightarrow\infty}{\lim}\lVert u_{n}\rVert_{p}=C$
exists, then $\lVert u\rVert_{p}=C$ and $\underset{n\rightarrow\infty}{\lim}\lVert Au_{n}\rVert_{q}=\lVert Au\rVert_{q}$. 
\end{lem}

Next, we can prove Theorem \ref{thm:main2} in a similar way. 
\begin{proof}[Proof of Theorem \ref{thm:main2}]
Let $A=\mid i\mid^{-\lambda}\ast:$ $\ell^{r}(\mathbb{Z}^{N})$$\longrightarrow$$\ell^{t}(\mathbb{Z}^{N})$.
By Corollary \ref{cor:lower bound2} and Lemma \ref{lem:proof2} we
obtain a maximizer $f$ for $K$. And by Lemma \ref{lem:exist g},
there exists a unique $g\in\ell^{s}(\mathbb{Z}^{N})$ with $\lVert g\rVert_{s}=1$
such that $(f,g)$ is a maximizing pair for (\ref{eq:hls1}). 
\end{proof}
Finally, we prove Corollary \ref{cor:p-Laplaican} in Section 1. 
\begin{proof}[Proof of Corollary \ref{cor:p-Laplaican}]
By Theorem \ref{thm:main1} there exists a minimizer $u$ for $S$.
Replacing $u$ by $\lvert u\rvert$, we know that $\lvert u\rvert$
is still a minimizer. Therefore, we get a non-negative solution $u$.
It follows from the Lagrange multiplier that $u$ is a solution of
equation (\ref{eq:p-Laplaican}). The maximum principle yields that
$u$ is positive. 
\end{proof}
We prove Corollary \ref{cor:E-R eq} in Section 1. 
\begin{proof}[Proof of Corollary \ref{cor:E-R eq}]
Define the functional 
\[
J(f,g)\coloneqq\sum_{\substack{i,j\in\mathbb{Z}^{N}\\
i\neq j
}
}\frac{f(i)g(j)}{\mid i-j\mid^{\lambda}}.
\]
From Theorem \ref{thm:main2} we know that there is a maximizing pair
$\left(f,g\right)$ for (\ref{eq:hls1}) under the constraint $\lVert f\rVert_{r}=\lVert g\rVert_{s}=1$.
A computation yields the Euler-Lagrange equation (\ref{eq:E-L eq})
for (\ref{eq:hls1}). Replace $\left(f,g\right)$ by $\left(\mid f\mid,\mid g\mid\right)$,
which is still a maximizing pair. The pair is positive by (\ref{eq:E-L eq}). 
\end{proof}
$\mathbf{\boldsymbol{\mathbf{Acknowledgements}\mathbf{}\text{:}}}$
The authors would like to thank Genggeng Huang for helpful discussions
and suggestions. B.H. is supported by NSFC, no.11831004 and no. 11926313.

 \bibliographystyle{plain}
\bibliography{sobolev_ref}

\end{document}